\documentclass[11pt]{amsart}
 \usepackage{graphicx}
 \usepackage{amssymb}
 \usepackage{amsmath}
 \usepackage{amsthm}
 \usepackage{mathtools}
 \usepackage{tikz}
 \usepackage{psfrag}
 \usepackage{xcolor}

\newtheorem{thm}{Theorem}[section]
\newtheorem{lem}[thm]{Lemma}
\newtheorem{pro}[thm]{Proposition}

\theoremstyle{definition}
\newtheorem{defn}{Definition}[section]
\newtheorem{remk}{Remark}[section]

\newcommand{\M}{\mathbb M}

\newcommand{\B}{\mathbf B}
\newcommand{\R}{\mathbb R}

\newcommand{\f}{\mathbf f}
\newcommand{\0}{\mathbf 0}

\newcommand{\tr}{\operatorname{tr}}
\newcommand{\rank}{\operatorname{rank}}

\newcommand{\dv}{\operatorname{div}}

\newcommand{\cof}{\operatorname{cof}}
\newcommand{\adj}{\operatorname{adj}}

\makeatletter
\numberwithin{equation}{section}
\makeatother

\begin{document}

 \author{Baisheng Yan}
 \address{Department of Mathematics\\ Michigan State University\\ East Lansing, MI 48824, USA}
   \email{yanb@msu.edu}

\title[Lipschitz solutions for polyconvex gradient flows]{Convex integration for diffusion equations, II: Lipschitz solutions for  polyconvex gradient flows} 


\begin{abstract} In this sequel to the  paper \cite{Y}, we construct certain smooth strongly polyconvex functions $F$ on $\M^{2\times 2}$ such that $\sigma=DF$ satisfies the Condition (OC) in that paper.    As a result,   we show that the initial-boundary value problem  for  the gradient flow of such polyconvex energy functionals   is  highly ill-posed even for some smooth initial-boundary data in the sense  that  the problem possesses a weakly* convergent sequence  of  Lipschitz weak solutions whose   limit is  not a weak solution.   
  \end{abstract}

\maketitle

\section{Introduction}

Let $m,n\ge 2$ be integers and $\M^{m\times n}$ be the space of $m\times n$ real matrices.  Given a  function $F\colon \M^{m\times n}\to \R$, we consider the energy functional 
 \begin{equation}\label{energy}
 \mathcal E(u)=\int_\Omega F(D u (x))\,dx, 
 \end{equation}
where  $\Omega\subset \R^n$ is a bounded domain with Lipschitz boundary $\partial\Omega$ and $u\colon \Omega\to \R^m.$ Here, if $u =(u^1,\dots,u^m),$  $Du =(\frac{\partial u^i}{\partial x_k})$ denotes the  Jacobian matrix of $u.$  
It is well-known that weak lower semicontinuity of  functional  $\mathcal E$ in a Sobolev space  is equivalent to   {\em Morrey's quasiconvexity} of function $F;$ see, e.g., {Acerbi \& Fusco} \cite{AF},  {Ball} \cite{Ba}, {Dacorogna} \cite{D}  and {Morrey} \cite{Mo}. Also closely related to  regularity of the energy minimizers of $\mathcal E$ is a stronger quasiconvexity condition (see, e.g., Evans \cite{Ev});
we say that  $F$ is  {\em strongly quasiconvex} if, for some constant $\nu>0$, 
 \begin{equation}\label{qx}
 \int_\Omega (F(A+D\phi(x))-F(A))dx \ge  \frac{\nu}{2} \int_\Omega |D\phi(x)|^2 dx
 \end{equation}
 holds for all  $A\in \M^{m\times n}$ and $\phi\in C_c^\infty(\Omega;\R^m).$
(If $\nu=0$, this condition becomes  the usual Morrey's quasiconvexity.) Since  $m,n \ge 2,$ it is well-known that a quasiconvex  function   may not be convex. 

If $F$ is $C^1$, then condition (\ref{qx}) implies 
\begin{equation}\label{mono-0}
\langle DF(A+p\otimes \alpha)-DF(A),\, p\otimes \alpha\rangle   \ge \nu |p|^2 |\alpha|^2
\end{equation}
for all $A\in \M^{m\times n}, p\in \R^m,$ and $\alpha\in \R^n.$ Note  that, for $C^2$ functions $F$,  condition (\ref{mono-0}) is equivalent to the uniform strong {\em Legendre-Hadamard} condition:
 \begin{equation}\label{LH}
 \sum_{i,j=1}^m\sum_{k,l=1}^n \frac{\partial^2 F(A)}{\partial a_{ik}\partial a_{jl}} p_ip_j\alpha_k\alpha_l \ge \nu |p|^2|\alpha|^2. 
 \end{equation} 
Furthermore,  for a $C^1$ function $F$  satisfying  certain natural growth conditions,  any  minimizer of energy $\mathcal E$ in a Dirichlet class  is a weak solution of the   {\em Euler-Lagrange system}:
\begin{equation}\label{ellip-00}
\dv DF(Du)=0\quad \mbox{in $\Omega.$}
\end{equation}

If $F$ is  a $C^2$ strongly quasiconvex function, then it is well-known   that every   Lipschitz minimizer  of $\mathcal E$ is of 
$C^{1,\alpha}(\Omega_0;\R^m)$ for all $0<\alpha<1$ on some open  set $\Omega_0$ with $| \Omega\setminus \Omega_0|=0;$  see  Evans  \cite{Ev}  and Fusco \& Hutchinson \cite{FH}. In a sharp contrast to such partial regularity of energy minimizers,  M\"uller and \v Sver\'ak  \cite{MSv2} proved  that,  for certain smooth (at least $C^2$) strongly quasiconvex  functions  $F$   on $ \M^{2\times 2},$
 the Euler-Lagrange system  (\ref{ellip-00}) has  Lipschitz weak solutions that are nowhere $C^1$ in  $\Omega;$ such  a result  has been extended  by Sz\'ekelyhidi   \cite{Sz1} to certain  {\em strongly polyconvex}  functions  
 $ F(A)=\frac{\nu}{2}|A|^2 +G(A,\det A) $ on $ \M^{2\times 2},$ where  $G(A,s)$ is smooth and convex on $(A,s)\in \M^{2\times 2}\times \R.$   We refer to \cite{Be, CY, KM} for further related partial regularity results.

In this paper, we study a parabolic  companion   of  system (\ref{ellip-00}).  More specifically, given $T>0$ and  initial function $u_0\colon \bar\Omega\to\R^m$, we study the initial-boundary value problem for  the $L^2$ {\em  gradient flow} of  energy $\mathcal E$ given by
\begin{equation}\label{ibvp-1}
\begin{cases}  u _t=\dv DF(Du ) \;\;\; \mbox{in $\Omega_T=\Omega \times (0,T)$,}
\\
\hfill u (x,t)=u_0(x) \;  (x\in \partial \Omega, \, 0<t<T),\;\; u (x,0)=u_0(x)& (x\in \Omega).
\end{cases}
\end{equation}

For convex  functions $F$,  standard monotonicity and semigroup methods of parabolic  theory would  apply to problem (\ref{ibvp-1}) (see {Br\'ezis} \cite{Br}); in this case,  one  easily sees that (\ref{ibvp-1}) possesses at most one weak solution. 

There is no  general   theory on  solvability of  problem (\ref{ibvp-1}) under the sole condition (\ref{mono-0}).
However, for  certain quasiconvex  functions  $F,$   (\ref{ibvp-1})  may be studied by a {\em time-discretization} method based on  recursively minimizing  the energies 
\[
\mathcal I_{j+1}^{h}(u) = \mathcal E(u)+ \frac{1}{2h} \|u-u^h_{j}\|^2_{L^2(\Omega)},\quad h>0,\;  j=0,1,\dots,
\]
where $u^h_{j}$ is a minimizer of $\mathcal I_{j}^{h}$ with given boundary data $u_0.$ We refer to  {Ambrosio, Gigli \& Savar\'e} \cite{AGS} for  general study of  such a method. In general,  when $h\to 0$, the  time-discretization scheme only produces the so-called  {\em generalized minimizing movements}  (see \cite{AGS}) or the  {\em Young measure solutions} to   (\ref{ibvp-1}) (see, e.g.,  {Kinderlehrer \& Pedregal} \cite{KP} and {Demoulini} \cite{De}). 
Existence of  the true weak solutions  to  (\ref{ibvp-1})  remains essentially open even for general  polyconvex functions  $F;$ see {Evans, Savin \& Gangbo} \cite{ESG} for a  special case of $F=\Phi(\det A)$. The same existence issue has also remained open for {\em elastodynamics} problems as studied in \cite{DST1,DST2}, but see \cite{KK} for a recent result on  existence of the true weak solutions for one-dimensional nonconvex elastodynamics.

 A parabolic problem similar  to problem (\ref{ibvp-1})  has been studied through the stationary elliptic system (\ref{ellip-00}) by M\"uller,  Rieger and  \v Sver\'ak  in \cite{MRS}, where they proved  that for certain smooth quasiconvex functions $F$, for all $\epsilon>0$ and $\alpha\in(0,1)$,  the elliptic system (\ref{ellip-00}) in $\Omega\subset \R^2$ possesses a family of   Lipschitz but nowhere $C^1$  weak solutions $ u (\cdot,t)$ for each $t\in (0,T)$ satisfying  $u (\cdot,0)=0,\, u (x,t)=0 \, (x\in \partial \Omega)$ and  $\| u _t\|_{C^\alpha(\Omega_T)}<\epsilon.$  
Such a function $u= u (x,t)$ is thus  a Lipschitz weak solution to  the  problem
\begin{equation}\label{ibvp-0}
\begin{cases}  u _t-\dv DF(D u )=f \quad  \mbox{in $\Omega_T$,}
\\
 u (x,t)=0 \; (x\in \partial \Omega, \, 0<t<T),\;\;  u (x,0)=0 \;  (x\in \Omega),
\end{cases}
\end{equation}
with  right-hand side $f\equiv u _t\in C^\alpha(\Omega_T;\R^2)$ satisfying $\|f\|_{C^\alpha}<\epsilon.$  However,   their method  cannot be adapted  to  the gradient flow problem (\ref{ibvp-1}), where $f=0$.   

The main purpose of the present paper is establish  that even for smooth strongly polyconvex functions $F$  and smooth initial-boundary functions $u_0$ the initial-boundary value problem (\ref{ibvp-1})  may be  highly ill-posed. Our main result is the following nonuniqueness and instability theorem.

\begin{thm}\label{thm-main-0} There exist  smooth strongly polyconvex functions  $F$ on $ \M^{2\times 2}$ and smooth functions $u_0$ such that problem  $(\ref{ibvp-1})$  possesses a sequence  of  Lipschitz weak solutions that converges weakly* to a function which is   
not  a Lipschitz  weak solution itself.
\end{thm}

We  point out that  in the theorem  one can choose $u_0(x)=Ax$ for some $A\in \M^{2\times 2}.$ In this case,  quasiconvexity condition (\ref{qx}) and the standard energy identity for $C^2$ solutions imply  that
 the only $C^2$ solution to problem (\ref{ibvp-1}) is $u(x,t)\equiv Ax;$ thus, the Lipschitz weak solutions  consisting of the  sequence  stated in  the theorem will   eventually be all distinct and   non-$C^2$ on $\bar\Omega_T,$ yielding infinitely many nonclassical, nonstationary  Lipschitz weak solutions to problem (\ref{ibvp-1}). 
 
 Theorem \ref{thm-main-0} follows from the general  results  of Yan \cite[Theorem 1.3 and Corollary 1.4]{Y}  once we have constructed the smooth strongly polyconvex functions $F\colon \M^{2\times 2}\to \R$  with $\sigma=DF$ satisfying the Condition (OC) given by \cite[Definition 3.2]{Y}.  Our construction (see  Proposition \ref{linprog-lem})  relies  on finding a  special $\tau_5$-configuration supported by  a smooth strongly polyconvex function. The search for such  special $\tau_5$-configuration and polyconvex function is accomplished  similarly as in  Sz\'ekelyhidi   \cite[Lemma 3]{Sz1}; however, in our case, we need to solve a more restrictive linear programming problem and rely  on  more specific results (see  Lemmas \ref{MATLAB-lem-1} and \ref{MATLAB-lem-2})  obtained by some necessary   MATLAB  computations.

 \section{Special $\tau_5$-configuration supported by a polyconvex function}\label{def-T5}

We refer to Yan \cite[Sections 2 and 3]{Y} for certain  definitions and notations.    In particular,  for $N\ge 3,$  let $ M'_N$ be the set of  $T_N$-configurations $(X_1,\dots,X_N)$ in $\M^{4\times 2}$ whose determining rank-one matrices are given by $C_j= \begin{pmatrix}p_j\\(\alpha_j\cdot \delta) q_j\end{pmatrix}\otimes  \alpha_j,$ where $p_j, q_j,\alpha_j,\delta \in\R^2$, $\alpha_j\ne 0$, at least three of $\alpha_j$'s are  {\em mutually non-collinear}, and  
\begin{equation}\label{tn-1}
 \sum_{j=1}^N p_j\otimes \alpha_j= 0,\quad  \sum_{j=1}^N q_j\otimes \alpha_j\otimes \alpha_j=0.
\end{equation}   
Let $J=\begin{pmatrix}0&-1\\1&0\end{pmatrix}$ and  let $\mathcal L\colon \M^{2\times 2}\times (\R^2)^2\to \M^{4\times 2}$ be defined  by
\begin{equation}\label{mapL}
\mathcal L([A,(b^i)])=\begin{bmatrix}A\\BJ\end{bmatrix} \quad \forall\, B=(b^i_k)\in \M^{2\times 2}.
\end{equation}
Define $\mathcal M'_N=\mathcal L^{-1}(M_N')$  to be the  set of  {\em special $\tau_N$-configurations} in $\M^{2\times 2}\times (\R^2)^2$ (see \cite[Definition 3.3]{Y}).

Consider a  function  $F\colon \M^{2\times 2}\to \R$ given by
$F(A)=\frac{\epsilon}{2}|A|^2 + G(A,\det A),$   where $G=G(A,s)$ is  smooth  on $\M^{2\times 2}\times \R.$  Then 
  \[
  DF(A)=\epsilon A+G_A(\tilde A) + G_s(\tilde A)\cof A,
  \]
  where $\tilde A=(A,\det A)$ and 
$
\cof A =\begin{pmatrix} a_{22}&-a_{21}\\-a_{12}&a_{11}\end{pmatrix}$ for $A=\begin{pmatrix} a_{11}&a_{12}\\a_{21}&a_{22}\end{pmatrix}.$ 
Let $\mathbb K_F=\{[A, (DF(A))]: A\in \M^{2\times 2}\}$ be the graph of $DF$ and define 
  \begin{equation}\label{ms-K}
K_F=\mathcal L(\mathbb K_F)= \left\{\begin{bmatrix} A\\DF(A)J\end{bmatrix} : A\in \M^{2\times 2}\right \}.
\end{equation}

Let $K=K_F$ and  $K_N=K\times\dots \times K$ ($N$ copies). 
Suppose $(X_1,\dots,X_N)\in M'_N\cap K_N$ with  $X_j=\begin{bmatrix}A_j\\B_j\end{bmatrix}$. Then
\begin{equation}\label{eq-cof}
\epsilon A_j+G_A(\tilde A_j) + G_s(\tilde A_j)\cof A_j=-B_jJ \quad (j=1,\dots,N).
\end{equation}

As in \cite{Sz1}, if we are given $c_j, d_j \in \R$ and $Q_j\in  \M^{2\times 2}$ for $j=1,\dots,N$, then it is well known that there exists a smooth convex function $G\colon \M^{2\times 2}\times \R\to \R$ satisfying $G(\tilde A_j)=c_j,\,G_A(\tilde A_j)=Q_j$ and $G_s(\tilde A_j)=d_j$ provided the following condition holds:
\begin{equation}\label{ineq-1}
c_j-c_i>\langle Q_i, A_j-A_i\rangle +d_i (\det A_j-\det A_i) \quad \forall\, i\ne j.
\end{equation}
If $G_A(\tilde A_j)=Q_j$ and $G_s(\tilde A_j)=d_j$ also satisfy  (\ref{eq-cof}), then   (\ref{ineq-1}) becomes
\begin{equation}\label{ineq-2}
c_i-c_j+d_i\det(A_i-A_j)+\langle A_i-A_j,B_iJ\rangle < -\epsilon \langle A_i,A_i-A_j\rangle\;\; \forall\, i\ne j.
\end{equation}
In particular, if
\begin{equation}\label{ineq-3}
c_i-c_j+d_i\det(A_i-A_j)+\langle A_i-A_j,B_iJ\rangle < 0 \;\; \forall\,i\ne j,
\end{equation}
then we can choose $\epsilon>0$ sufficiently small so that (\ref{ineq-2}) is   satisfied.

\begin{pro}\label{linprog-lem} Let $N=5.$ Then there exists  $(X_1,\dots,X_5)\in M'_5$  with  $X_j=\begin{bmatrix}A_j\\B_j\end{bmatrix}$ such that $(\ref{ineq-3})$ holds for some constants $c_1,\dots,c_5$ and $d_1,\dots,d_5.$
\end{pro}

\begin{proof} 
Suppose $(X_1,\dots,X_5)\in M'_5$ is  determined by $P, \,C_i$ and $\kappa_i$ such that $(\ref{ineq-3})$ holds for some constants $c_1,\dots,c_5$ and $d_1,\dots,d_5.$ Then, by adjusting the  constants, we may assume $P=0,\, c_1=0.$ Let 
\[
C_j=\begin{pmatrix} p_j\\(\alpha_j\cdot\delta)q_j\end{pmatrix} \otimes \alpha_j\quad (j=1,\dots,5),
\]
where  $\alpha_j\ne 0$ and at least three $\alpha_j$'s are mutually noncollinear. 
We attempt  to find such a $T_5$-configuration with $\delta=(1,1)$ and
\[
\alpha_1=(-1,0),\;\alpha_2=(0,-1),\; \alpha_3=(1,z_3), \; \alpha_4=(1,z_4),\;\alpha_5=(y_5,1), \]
where $z_3\ne 0.$ 
From condition (\ref{tn-1}), we  solve $p_1,\,p_2$ as    linear combinations of $p_3,p_4,p_5$ and solve $q_1,q_2,q_3$ as   linear combinations of $q_4,q_5$ to obtain
\begin{equation}\label{no-free-0}
\begin{cases}
p_1=p_3+p_4+y_5p_5, \;\; p_2=z_3p_3+z_4 p_4+p_5,\\
q_1=(z_4/z_3-1)q_4 +(y_5/z_3-y_5^2)q_5,\\
q_2=(z_3z_4-z_4^2 )q_4+(y_5z_3-1)q_5, \;\;  q_3=-(z_4/z_3)q_4-(y_5/z_3)q_5.
\end{cases}
\end{equation}
Then $\delta_{ij}=\det(A_i-A_j)$ becomes a function of $y_5,z_3,z_4, \kappa_i, p_3,p_4$ and $p_5$, while $\mu_{ij}=\langle A_i-A_j,B_iJ\rangle$ becomes  a linear function of $q_4,q_5$ with coefficients depending on 14 free-variables $y_5,z_3,z_4,\kappa_i, p_3,p_4$ and $p_5$. With a fixed choice of $(y_5,z_3,z_4,\kappa_i, p_3,p_4,p_5)\in \R^{14},$ condition (\ref{ineq-3}) becomes a system of 20 linear inequalities on  13  variables:
$c_2,\dots,c_5,\, d_1,\dots,d_5,\, q_4,\, q_5.$ The matrix  representing  the LHS of (\ref{ineq-3}) is a $20\times 13$ matrix $A$  depending  on 
 $(y_5,\, z_3,\, z_4,\, \kappa_i, \, p_3,\, p_4, \, p_5).$  Using the {\em linear programming}  by MATLAB, after numerous  trials on different choices of $(y_5,\, z_3,\, z_4,\, \kappa_i, \, p_3,\, p_4, \, p_5)$ for the matrix $A$, the linear inequality  $AX<0$ finally becomes feasible  with 
\[
\begin{cases}
y_5=2,\; z_3=1,\; z_4=4, \\
p_3=(1,0),\; p_4=(1,1),\; p_5=(0,1),\\
\kappa_1= 2,\; \kappa_2=3,\;  \kappa_3=4,\;  \kappa_4=3,\;  \kappa_5=2.\end{cases}
\]
With this choice of parameters,  a feasible solution $X$ for $AX<0$ gives  
\[
 q_4=(-19,0),\quad q_5=(-63,-82).
\] 
Hence by (\ref{no-free-0}) we have
\[
p_1 =(2,3),\;  p_2 =(5,5),\; \; q_1 =(69,164),\; q_2 =(165,-82),\; q_3 =(202,164),
\]
and therefore the  rank-one matrices $C^0_j=\begin{pmatrix} p_j\\(\alpha_j\cdot\delta)q_j\end{pmatrix}\otimes \alpha_j$ are  given by
\[
C^0_1=\begin{pmatrix}-2&0\\-3&0\\69&0\\164&0\end{pmatrix},\;\; C^0_2=\begin{pmatrix}0&-5\\0&-5\\0&165\\0&-82\end{pmatrix},\;\; 
C^0_3=\begin{pmatrix}1&1\\0&0\\404&404\\328&328\end{pmatrix},
\]
\[
C^0_4=\begin{pmatrix}1&4\\1&4\\-95&-380\\0&0\end{pmatrix},\;\; 
C^0_5=\begin{pmatrix} 0&0\\2&1\\-378&-189\\-492&-246\end{pmatrix}.
\]
The corresponding $T_5$-configuration   $(X^0_1,\dots,X^0_5)\in M'_5$ is illustrated in Fig.\,\ref{fig1}, where $X^0_i=\begin{pmatrix}A^0_i\\B^0_i\end{pmatrix}$ with
\[
A^0_1=\begin{pmatrix} -4&0\\-6&0\end{pmatrix},\; A^0_2=\begin{pmatrix} -2&-15\\-3&-15\end{pmatrix},\; A^0_3=\begin{pmatrix} 2&-1\\-3&-5\end{pmatrix},\]
\[
A^0_4=\begin{pmatrix} 2&8\\0&7\end{pmatrix},\; A^0_5=\begin{pmatrix} 0&0\\2&1\end{pmatrix}.
\]
\[
B^0_1=\begin{pmatrix} 138&0\\328&0\end{pmatrix},\; B^0_2=\begin{pmatrix} 69&495\\164&-246\end{pmatrix},\; B^0_3=\begin{pmatrix} 1685&1781\\1476&1230\end{pmatrix},\]
\[
B^0_4=\begin{pmatrix} 188&-571 \\492 &246\end{pmatrix},\; B^0_5=\begin{pmatrix} -378&-189\\-492&-246\end{pmatrix}.
\]
For these $A^0_j$ and $B^0_j$, it is straightforward (with the aid of  MATLAB) to verify that  condition (\ref{ineq-3}) is satisfied with the constants:
\[
(c^0_1,\dots,c^0_5,d^0_1,\dots,d^0_5)=(0,-3650,-3318,5044,580,58,-7.5,772,57,376).
\]
This completes the proof.
 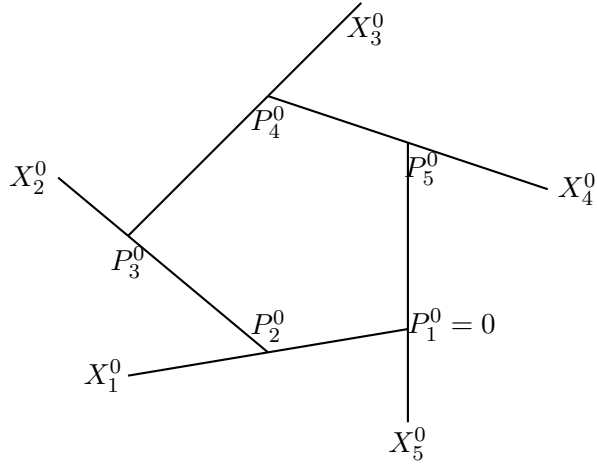
\begin{figure}[ht]
\begin{center}
\begin{tikzpicture}[scale =.62]
\draw[thick] (-5,-2)--(1,-1);
  \draw(-4.9,-2) node[left]{$X^0_1$};
      \draw(0.1,6) node[below]{$X^0_3$};
   \draw(3.1,-0.9) node[left]{$P^0_1=0$};
    \draw(1.3,3) node[below]{$P^0_5$};
         \draw(-5,1) node[below]{$P^0_3$};
          \draw(-2,4) node[below]{$P^0_4$};
     \draw[thick] (1,-1)--(1,3);
  \draw[thick] (1,-1)--(1,-3);
  \draw[thick] (-2,-1.5)--(-5,1);
   \draw[thick] (-6.5,2.25)--(-5,1);
   \draw[thick] (-5,1)--(-2,4);
     \draw[thick] (0,6)--(-2,4);
    \draw[thick] (-2,4)--(1,3);
       \draw[thick] (1,3)--(4,2);
     \draw(1,-2.9) node[below]{$X_5^0$};
          \draw(-2,-1.5) node[above]{$P^0_2$};
    \draw(4,2) node[right]{$X_4^0$};
   \draw(-6.5,2.25) node[left]{$X^0_2$};
\end{tikzpicture}
\end{center}
\caption{The special $T_5$-configuration $(X^0_1,\dots,X^0_5)$ with pentagon $[P_1^0P_2^0P_3^0P_4^0P_5^0]$ constructed  in Proposition \ref{linprog-lem}. } 
\label{fig1}
\end{figure}
\end{proof}

 \begin{remk} \label{remk-61}   Our construction   is substantially  different from that of  \cite[Lemma 3]{Sz1} because the rank-one matrices $C_j^0$ satisfy more restrictive conditions. The special and more restrictive $T_5$-configuration  constructed above  provides a correct $T_5$-configuration  for \cite[Example 1]{Sz1} as the example of  $T_5$-configuration there was  incorrectly printed.  
 \end{remk}

\begin{pro}\label{poly-F0} Let $(X_1^0,\dots,X_5^0)\in M'_5$ be the special $T_5$-configuration constructed above.  Then for  all sufficiently small $\epsilon>0$ there exists a smooth convex function $G\colon \M^{2\times 2}\times \R\to \R$  such that the strongly polyconvex function $F_0(A)=\frac{\epsilon}{2}|A|^2 +G(A,\det A)$ satisfies that  $X_j^0 \in K_{F_0}$ for $j=1,\dots,5.$
\end{pro}
\begin{proof}  Let $c_1^0,\dots,c_5^0$ and $d_1^0,\dots,d_5^0$ be the constants determined above. Then (\ref{ineq-2}) is satisfied for all sufficiently small $\epsilon>0.$  For such an $\epsilon>0$, let $Q^0_j=-\epsilon A^0_j-B^0_jJ -d^0_j\cof A^0_j.$  Following the proof of \cite[Lemma 3]{Sz1},   there exists a smooth convex function 
$G\colon \M^{2\times 2}\times \R\to \R$ such that $G(\tilde A^0_j)=c_j^0$, $G_A(\tilde A_j^0)=Q_j^0$ and $G_s (\tilde A_j^0)=d_j^0$ for $j=1,\dots,5.$ Let $F_0(A)=\frac{\epsilon}{2}|A|^2 +G(A,\det A).$ Then $F_0$ is smooth and strongly polyconvex; clearly, for each $j=1,\dots,5$, one has $DF_0(A^0_j)=\epsilon A^0_j+G_A(\tilde A^0_j) + G_s (\tilde A^0_j)\cof A^0_j=-B_j^0J,$ which proves $X_j^0\in K_{F_0}.$
\end{proof}

\section{The perturbations of $(X^0_1,\dots,X^0_5)$ and $F_0$}

Let $(X^0_1,\dots,X^0_5)\in M'_5$ and strongly polyconvex function $F_0$ be determined as  in Propositions \ref{linprog-lem} and \ref{poly-F0}  above. The  parameters of    $(X^0_1,\dots,X^0_5)$ are listed as follows for later usage:
\begin{equation}\label{free-para}
\begin{cases}
P^0=0,\; \delta^0=(1,1),\\
\alpha_1^0=(-1,0), \alpha_2^0=(0,-1), \alpha^0_3=(1,1), \alpha_4^0=(1,4), \alpha_5^0=(2,1),\\
 p_1^0=(2,3),\; p_2^0=(5,5),\; p_3^0=(1,0),\; p_4^0=(1,1),\;p_5^0=(0,1),\\
 q_1^0=(69,164),\, q_2^0=(165,-82),\, q_3^0=(202,164),\\
  q_4^0=(-19,0),\, q_5^0=(-63,-82),\\
\kappa_1^0= 2,\; \kappa_2^0= 3,\; \kappa_3^0= 4,\; \kappa_4^0=3,\; \kappa_5^0=2.
\end{cases}
\end{equation}

We  aim to construct the required strongly polyconvex functions $F$  as a perturbation of   function $F_0,$  with the set $K_F$ containing $(X^0_1,\dots,X^0_5)$ and all suitable perturbations of $(X^0_1,\dots,X^0_5).$

First of all, to perturb the function $F_0$, let $B_1(0)$ be the unit ball in $\M^{2\times 2}$ and  $\zeta\in C^\infty_c(B_1(0))$ be such that
$0\le \zeta(A)\le 1$ and $\zeta(0)=1.$ Given $r>0$ and tensor $H=(H^{pqij})$ with $H^{pqij}=H^{ijpq}\in\R$ for all $i,j,p,q\in\{1,2\}$ (thus $H$ can be viewed as a $4\times 4$ symmetric matrix), define the function
\[
V_{H,r}(A)= \frac12 \zeta(A/r) \sum_{i,j,p,q\in\{1,2\}} H^{ijpq}a_{ij}a_{pq}
\]
for $A=(a_{ij})\in\M^{2\times 2}.$  Then $V_{H,r}\in C^\infty_c(\M^{2\times 2})$ has support in $B_r(0),$ and  
\begin{equation}\label{cut-off}
\begin{cases}
V_{H,r}(0)=0,\; DV_{H,r}(0)=0,\; D^2 V_{H,r}(0)=H,\\
|D^2 V_{H,r}(A)|\le C |H| \quad \forall A\in \M^{2\times 2},
\end{cases}
\end{equation}
where $C$ is a constant independent of $H,r.$

Let 
$r_0=\min_{i\ne j} |A_i^0-A_j^0|>0$.  
The functions $F$ will be constructed as a perturbation of $F_0$ of the form:
\begin{equation}\label{def-F}
F(A)=F_0(A)+\sum_{j=1}^5 V_{\tilde H_j,r_0}(A-A_j^0),
\end{equation}
with $\tilde H_1,\dots,\tilde H_5$ to be chosen later.   
Note that
\begin{equation}\label{prop-F}
\begin{cases} DF(A_j^0)=DF_0(A_j^0),\\
D^2 F(A_j^0)=D^2 F_0(A_j^0)+\tilde H_j  \end{cases} \quad (j=1,\dots,5);
\end{equation}
thus, the specific $T_5$-configuration $(X_1^0,\dots,X_5^0)$ also lies on $K_F.$  
Furthermore, $F$ is strongly polyconvex if 
$\sum_{j=1}^5 |D^2V_{\tilde H_j,r_0}(A)|<\epsilon/2,$ which will be satisfied if 
\begin{equation}\label{def-F2}
\sum_{j=1}^5 |\tilde H_j|<\frac{\epsilon}{2C}, \;\; \mbox{with $C$ being the constant in (\ref{cut-off}).}
\end{equation}

We now  perturb  the $T_5$-configuration  $(X_1^0,\dots,X_5^0)$ around  the parameters in  (\ref{free-para}) by the  following 28 parameters (notice   that  we do not perturb $\delta_0$):
\[
\begin{cases}
Q\in \M^{4\times 2}\cong \R^8, \\
\alpha_1=(-1,z_1), \alpha_2=(y_2,-1), \alpha_3=(1,z_3), \alpha_4=(1,z_4), \alpha_5=(y_5,1),\\
 p_3=(p_{31},p_{32}),\; p_4=(p_{41},p_{42}),\;p_5=(p_{51},p_{52}), \\
q_4=(q_{41},q_{42}),\;q_5=(q_{51},q_{52}),\\
\kappa_1,\;\kappa_2 ,\;\kappa_3,\;\kappa_4,\;\kappa_5.
\end{cases}
\]
Let $
Y=(z_1,y_2,z_3,z_4,y_5, p_{3},p_{4},p_{5}, q_{4}, q_{5}, \kappa_1,\dots,\kappa_5)\in\R^{20},$ and define  $p_1,p_2$ and  $q_1,q_2,q_3$  according to  (\ref{tn-1})  by
\[
\begin{split}
&p_1=\frac{y_2z_3+1}{1-y_2z_1}p_3+\frac{y_2z_4+1}{1-y_2z_1}p_4+\frac{y_2+y_5}{1-y_2z_1}p_5,\\
&p_2=\frac{z_1+z_3}{1-y_2z_1}p_3+\frac{z_1+z_4}{1-y_2z_1}p_4+\frac{y_5z_1+1}{1-y_2z_1}p_5,\\
&q_1=\frac{(y_2 z_4 + 1)(z_3 - z_4)}{(z_1 + z_3)(y_2z_1 - 1)}q_4+\frac{(y_2 + y_5)(y_5z_3 - 1)}{(z_1 + z_3)(y_2z_1 - 1)}q_5,\\
&q_2=-\frac{(z_1 + z_4)(z_3 - z_4)}{(y_2z_1 - 1)(y_2z_3 + 1)}q_4-\frac{(y_5z_1 + 1)(y_5z_3 - 1)}{(y_2z_1 - 1)(y_2z_3 + 1)}q_5,\\
&q_3=-\frac{(z_1 + z_4)(y_2z_4 + 1)}{(z_1 + z_3)(y_2z_3 + 1)}q_4-\frac{(y_2 + y_5)(y_5z_1 + 1)}{(z_1 + z_3)(y_2z_3 + 1)}q_5.
\end{split}
\]

For $j=1,\dots,5,$ define  $C_j=C_j(Y)=\begin{pmatrix}p_j\\(\alpha_j\cdot \delta^0)q_j\end{pmatrix}\otimes \alpha_j$,  and for  $\nu=1,\dots,5,$   define
\[
 \begin{split}
&Z^\nu_1(Y)  = \kappa_\nu C_\nu,\\
&Z^\nu_2 (Y)  =C_\nu +\kappa_{\nu+1} C_{\nu+1},\\
& Z^\nu_3 (Y) =C_\nu +C_{\nu+1}+\kappa_{\nu+2} C_{\nu+2},  \\
&Z^\nu_4 (Y)  =C_\nu +C_{\nu+1}+C_{\nu+2}+\kappa_{\nu+3} C_{\nu+3}, \\
&Z^\nu_5 (Y)  = C_\nu +C_{\nu+1}+C_{\nu+2}+ C_{\nu+3}+\kappa_{\nu+4} C_{\nu+4}.
\end{split}
\]
Henceforth,  {\em the indices  $\nu, j$ appearing in  $Z^\nu_j$, $C_j$ and $\kappa_j$ are always taken modulo 5.}  For all $\nu$ and $j$,  define
   \begin{equation}\label{Q-switch-0}
  X^\nu_j(Y,Q)=Q+Z^\nu_j(Y).
 \end{equation}
 Then $(X^\nu_1(Y,Q),\dots,X^\nu_5(Y,Q))\in M'_5$ for all  $\nu$ and  $(Y,Q)$. Moreover,  
  \begin{equation}\label{Q-switch-3}
 X^\nu_j(Y^0,P^0_\nu)=X^0_{\nu+j-1} \quad  \forall\, \nu,\,   j \mod 5,
  \end{equation} 
  where $Y^0=(0,0,1,4,2,p_3^0,p_4^0,p_5^0,q_4^0,q_5^0,2,3,4,3,2)\in \R^{20} $ and $P_\nu^0$'s are the vertices of the pentagon as shown   in Figure \ref{fig1}.
  
  The following result is immediate.
  
\begin{lem}\label{rel-X} Let $C_j=C_j(Y)$ be defined as above. Define
\[
 \begin{split}
&P^\nu_1(Y,Q) = Q,\\
& P^\nu_2 (Y,Q) =Q+ C_\nu, \\
&P^\nu_3 (Y,Q) =Q+ C_\nu+ C_{\nu+1},  \\
&P^\nu_4 (Y,Q) =Q+ C_\nu+ C_{\nu+1}+C_{\nu+2},   \\ 
&P^\nu_5 (Y,Q) = Q+ C_\nu+ C_{\nu+1}+C_{\nu+2}+ C_{\nu+3}.  
\end{split}
\]
Then $X^\nu_{j}(Y,Q)=X^{\nu+i-1}_{j-i+1}(Y,P^\nu_{i}(Y,Q))$ for all $\nu,j,i \mod 5,$ with $j\ge i.$  
\end{lem}

In what follows, we  fix the identifications: $\M^{2\times 2}\cong \R^4$ and  $\M^{4\times 2}\cong \R^8$ as follows, in the same way as used in MATLAB:
\begin{equation}\label{identification}
\begin{pmatrix}x_{11}&x_{12}\\x_{21}&x_{22}\end{pmatrix}\cong \begin{pmatrix}x_{11}\\x_{21}\\x_{12}\\x_{22} \end{pmatrix},  \quad \begin{pmatrix}x_{11}&x_{12}\\x_{21}&x_{22}\\x_{31}&x_{32}\\x_{41}&x_{42}\end{pmatrix}\cong \begin{pmatrix}x_{11}\\x_{21}\\x_{31}\\x_{41}\\x_{12}\\x_{22}\\x_{32}\\x_{42}\end{pmatrix}.
\end{equation}
  
\begin{lem}\label{rk5} For each $\nu=1,\dots,5$, the $8\times 20$ matrix $\dfrac{\partial Z^\nu_1}{\partial Y}$ has rank less than or equal to 5.
\end{lem} 
\begin{proof} Write $Z^\nu_1(Y)=\gamma_\nu\otimes \alpha_\nu,$ where $\gamma_\nu=\kappa_\nu\begin{pmatrix}p_\nu\\(\alpha_\nu\cdot \delta^0)q_\nu\end{pmatrix}=(f_i(Y))\in \R^{4}.$ Set $\alpha_\nu=(a,b)\in\R^2,$ where one of $a,b$ is constant and the other is a variable in  $Y$. Thus, with the identification (\ref{identification}), as a $8\times 20$ matrix,
\[
\frac{\partial Z^\nu_1}{\partial Y}=
\begin{pmatrix}f_{1}\\f_{2}\\f_3\\f_{4}\\0\\0\\0\\0\end{pmatrix}\otimes  \frac{\partial a}{\partial Y}+\begin{pmatrix} 0\\ 0\\ 0\\0\\f_1\\f_2\\f_3\\f_{4}\end{pmatrix}\otimes  \frac{\partial b}{\partial Y}+\begin{pmatrix} a\\0\\0\\ 0\\b\\0\\0\\0 \end{pmatrix}\otimes  \frac{\partial f_1}{\partial Y}+\dots + \begin{pmatrix} 0\\0\\ 0\\ a\\0\\0\\0\\b\end{pmatrix}\otimes  \frac{\partial f_4}{\partial Y}.
\]
Since one of $\partial a/\partial Y$ or $\partial b/\partial Y$ vanishes, we easily have $\rank \frac{\partial Z^\nu_1}{\partial Y}\le 5.$
\end{proof}

\section{The construction of polyconvex functions $F$}

Suppose that $F$ is a function defined by  (\ref{def-F}) with $\tilde H_j$ to be chosen. To study the set $K_F$,  we define the function  $\Phi\colon \M^{4\times 2}\cong \R^8  \to \M^{2\times 2}\cong \R^4$  by
\begin{equation}\label{Phi-0}
\Phi(X)=DF(A) + BJ, 
\end{equation}
where $X=\begin{bmatrix}A\\B\end{bmatrix}\in\M^{4\times 2}\cong \R^8.$ Then $X\in  K_F$ if and only if $\Phi(X)=0.$ 

With the identification (\ref{identification}), we have
$A=PX$ and $BJ=EX$, where
\[
P=\begin{pmatrix}I&O&O&O\\   O&O&I&O\end{pmatrix}, \;\; E=\begin{pmatrix}O&O&O&I\\O&-I&O&O\end{pmatrix}
\]
with  $I=I_2$ ($2\times 2$ identity matrix) and $O=O_2$ ($2\times 2$ zero matrix).
Therefore, as a $4\times 8$ matrix, we have
\begin{equation}\label{D-Phi}
 D\Phi(X)=D^2 F(A)P+E, 
\end{equation}
where the tensor $D^2F(A)$ is viewed  as a symmetric $4\times 4$ matrix. Hence
\begin{equation}\label{constant-rank}
\rank(D\Phi(X))=4 \quad \forall\, X\in \M^{4\times 2}.
\end{equation}

 For $\nu=1,\dots,5$, we define the functions
\begin{equation}\label{Psi}
\Psi^\nu(Y,Q)=(\Phi(X^{\nu}_1(Y,Q)), \dots, \Phi(X^\nu_5(Y,Q))).
\end{equation}
Then $\Psi^\nu(Y^0,P_\nu^0)=0.$ By Lemma \ref{rel-X},  we have, component-wise, modulo 5, 
\begin{equation}\label{rel-X-4}
\Psi_j^\nu(Y,Q)=\Psi_{j-i+1}^{\nu+i-1}(Y,P^\nu_i(Y,Q)) \quad\forall\, j\ge i.
\end{equation}

To  study the equation $\Psi^\nu(Y,Q)=0$  near  $(Y^0,P_\nu^0),$ note that, by (\ref{Phi-0}) and (\ref{Psi}),  
the partial Jacobian matrix
$\dfrac{\partial\Psi^\nu}{\partial Y}(Y,Q)$   is a $20\times 20$ matrix given by
\begin{equation}\label{jacobi-0}
\frac{\partial\Psi^\nu}{\partial Y}(Y,Q) =
\begin{bmatrix} D\Phi(X^\nu_1)  \dfrac{\partial Z^\nu_1}{\partial Y}\\
 \vdots\\
  D\Phi(X^\nu_5)  \dfrac{\partial Z^\nu_5}{\partial Y} 
  \end{bmatrix}.
\end{equation}
Thus,  the matrix   $\frac{\partial\Psi^\nu}{\partial Y}(Y,Q)$ depends  on the Hessian tensors  
$D^2F(PX^\nu_k)$ $(k=1,\dots,5)$ and is otherwise   independent of $F$ and $Q$. 
Define 
\[
J_\nu=\det \frac{\partial\Psi^\nu}{\partial Y}(Y^0,P_\nu^0).
\]
Since $X^\nu_j(Y^0,P_\nu^0)=X^0_{\nu+j-1}$ for all $\nu,j=1,\dots,5,$ we have
\[
D^2 F(PX^\nu_j(Y^0,P_\nu^0)) \in \{D^2 F(A^0_1),\dots,D^2 F(A_5^0)\} \quad \forall\, \nu,j=1,\dots,5,
\]
and thus $J_\nu$
is a polynomial of $H_1=D^2 F(A^0_1),\dots,H_5=D^2 F(A_5^0)$ with coefficients  being independent of $F$.   We write this polynomial as 
\begin{equation}\label{def-j-nu}
J_\nu=j_\nu(H_1,H_2,H_3,H_4,H_5).
\end{equation}
 
 \begin{lem} \label{MATLAB-lem-1} The polynomial $j_\nu(H_1,\dots,H_5)$ is not identically zero for each $\nu=1,\dots,5.$ 
 \end{lem}
 \begin{proof}
 Given $s, t$, we consider the  tensors (with $I=I_2$, $O=O_2$):
 \[
 h_1(s)=\begin{pmatrix}sI&O\\O&I\end{pmatrix},\;\; h_2(t)=\begin{pmatrix}I&O\\O&tI\end{pmatrix}, \]
 and let $ g_\nu(s,t)=j_\nu(h_1(s),h_2(t),h_1(s),h_1(s),h_2(t)).$
 Then,   using the {\em jacobian} computations  by MATLAB,  we verify that
 \begin{equation}\label{nonzero-jacobian}
 g_1(1,0)\ne 0,\;  g_2(0,0)\ne 0,\; g_3(0,1)\ne0,\;  g_4(0,0)\ne 0,\; g_5(0,0)\ne 0.
 \end{equation}
Therefore, the polynomial $j_\nu(H_1,\dots,H_5)$ is not identically zero for each $\nu=1,\dots,5.$ 
\end{proof}

Since polynomial  $j_\nu(H_1,\dots,H_5)$ is not identically zero for each $\nu=1,\dots,5,$ we   select $(H_1^0,\dots,H_5^0)$ with the following properties:
\begin{equation}\label{ImFT-1}
\begin{cases} j_\nu(H^0_1,\dots,H^0_5)\ne 0 \;\;\; \forall \, \nu=1,\dots,5;  \\
 \mbox{the tensors $\tilde H_j=H_j^0-D^2 F_0(A_j^0)$ satisfy (\ref{def-F2})}.
 \end{cases}
\end{equation}

\begin{lem} \label{def-Y-nu} With $\tilde H_j$  chosen to satisfy $(\ref{ImFT-1})$, define $F$  by $(\ref{def-F})$ and  $\Psi^\nu$    by $(\ref{Psi}).$ Then  there exist  positive numbers $\rho,\delta$  and  smooth functions
\[
Y_\nu \colon B_\rho (P_\nu^0)\subset  \M^{4\times 2}\cong \R^8 \to B_\delta(Y^0)\subset \R^{20},
\]
 for   $\nu=1,\dots,5$, such that $Y_\nu(P_\nu^0)=Y^0$ and, for all $Y\in B_\delta(Y^0)$ and $Q\in B_\rho(P_\nu^0)$, it follows that
 \begin{equation}\label{ImFT-4}
\det \dfrac{\partial\Psi^\nu}{\partial Y}(Y,Q)\ne 0; \quad 
  \Psi^\nu(Y,Q)=0 \iff Y=Y_\nu(Q). 
    \end{equation}
  \end{lem}
  \begin{proof}
  Note that $\Psi^\nu(Y^0,P_\nu^0)=0$ and, by (\ref{ImFT-1}),
 \[
 \det \frac{\partial\Psi^\nu}{\partial Y}(Y^0,P_\nu^0) =  j_\nu(H^0_1,\dots,H^0_5)\ne 0\quad \forall\; \nu=1,\dots,5.
 \] 
Thus, the conclusion follows from  the implicit function theorem.
\end{proof}

\begin{lem}\label{lem-key-1} Let $z^\nu(Q)= Z^\nu_1(Y_\nu(Q))$  for $Q\in B_\rho(P_\nu^0)\subset\R^8.$  Then  all  eigenvalues of  matrix  $M=Dz^\nu(Q)\in \M^{8\times 8}$ are real numbers $\{-1,0,\mu_M\},$ where $\mu_M=4+\tr(M),$ with  $-1$ being an eigenvalue  of multiplicity at least 4, and  0 an eigenvalue   of multiplicity at least 3.
\end{lem}
\begin{proof} Clearly $M=Dz^\nu(Q)=\frac{\partial Z^\nu_1}{\partial Y} (Y_\nu(Q)) \frac{\partial Y_\nu}{\partial Q}(Q).$  Differentiate  the equation
$\Psi^\nu(Y_\nu(Q),Q)=0$ with respect to $Q$ to have
  \begin{equation} \label{jacobi-5}
 \frac{\partial Y_\nu}{\partial Q}(Q)= - \left [\Big (\frac{\partial \Psi^\nu}{\partial Y}\Big )^{-1} \frac{\partial \Psi^\nu}{\partial Q}\right](Y_\nu(Q),Q).
\end{equation}
Thus
\[
M = - \frac{\partial Z^\nu_1}{\partial Y} (Y_\nu(Q)) \left [\Big (\frac{\partial \Psi^\nu}{\partial Y}\Big )^{-1} \frac{\partial \Psi^\nu}{\partial Q}\right](Y_\nu(Q),Q).
\]
Let $R_j =D\Phi(X^\nu_j)= D\Phi(Q+Z^\nu_j(Y_\nu(Q)))\in \M^{4\times 8}$ and 
\begin{equation}\label{mx-SRT} S_j= \frac{\partial Z^\nu_j}{\partial Y} (Y_\nu(Q))\in \M^{8\times 20},\quad 
[T_1,\dots,T_5]= \left [\frac{\partial \Psi^\nu}{\partial Y}(Y_\nu(Q),Q)\right ]^{-1}, 
\end{equation}
where each $T_j\in \M^{20\times 4}.$ Then 
\[
\frac{\partial \Psi^\nu}{\partial Q}(Y_\nu(Q),Q)=\begin{bmatrix}R_1\\\vdots\\R_5\end{bmatrix},\quad 
\frac{\partial \Psi^\nu}{\partial Y}(Y_\nu(Q),Q)=\begin{bmatrix}R_1S_1\\\vdots\\R_5S_5\end{bmatrix},
\]
and thus
\begin{equation}\label{mx-M}
\begin{split}
R_1 S_1 T_1=I_4,\;\;  R_1S_1T_k&=0 \quad \forall \,  k\ne 1,\\
M=-(S_1T_1R_1+S_1T_2R_2&+\dots +S_1T_5R_5).\end{split}
\end{equation}
 Therefore  $R_1 M=-R_1$ and so $R_1(I_8+M)=0.$ By (\ref{constant-rank}),  $\rank R_1=4;$  thus   $\dim\ker (R_1)=4$ and hence  $\rank(I_8+M)\le 4,$ which gives
\begin{equation}\label{eigen-1}
\dim \ker (I_8+M)\ge 4.
\end{equation}
By Lemma \ref{rk5}, $\rank(S_1)\le 5;$ thus
 $\rank(M)\le 5$ and hence
 \begin{equation}\label{eigen-0}
  \dim \ker (M) \ge 3.
 \end{equation}
By (\ref{eigen-1}) and (\ref{eigen-0}), $M$ has $-1$ as an eigenvalue of multiplicity at least 4 and has 0 as an eigenvalue of multiplicity at least 3. Since the sum of all eigenvalues of $M$ is $\tr(M)$, it follows that  all  eigenvalues of $M$ consist of $\{-1,0, \mu_M\}.$
 \end{proof}

\begin{lem}\label{lem-key2} Let $M=Dz^\nu(Q)$  be defined as above. If $\mu_M\notin\{0,-1\},$ then $\rank[\adj(I-  \mu_M^{-1} M)]=1$ and, for  any $b\in \R^8$,
\begin{equation}\label{determinant}
\det (I- \mu_M^{-1} M+z^\nu\otimes b)=[\adj(I- \mu_M^{-1} M)z^\nu]\cdot b,
\end{equation}
where $\adj A$ is the adjoint  of  $A\in\M^{8\times 8}$ satisfying  $A(\adj A)=(\det A)I_8.$ 
\end{lem}
\begin{proof} Since $\mu_M\notin\{0, -1\},$ it must be a simple eigenvalue of $M$ and thus $\rank(\mu_M I-M)=7$; so, $\rank[\adj(I- \mu_M^{-1} M)]=1.$   Finally, (\ref{determinant}) follows from the formula: $\det (A+\alpha\otimes \beta)=\det A + (\adj A )\alpha \cdot \beta.$
\end{proof}

Let $M_\nu^0=Dz^\nu(P_\nu^0).$ Then, by (\ref{mx-SRT}) and (\ref{mx-M}), 
\begin{equation}\label{M_nu}
M_\nu^0=\frac{W(H^0_1,\dots,H^0_5)}{j_\nu(H^0_1,\dots,H^0_5)},
\end{equation}
where $H^0_j=D^2 F(A_j^0)$ ($j=1, \dots,5$), $W(H_1,\dots,H_5)$ is a $8\times 8$ matrix whose entries are polynomials of tensors $(H_1,\dots,H_5)$, and  $j_\nu(H_1,\dots,H_5)$ is the polynomial function  defined by  (\ref{def-j-nu}).  Again, both $W$ and $j_\nu$  are independent of the function $F$. 
Therefore,  both $\mu_{M_\nu^0}(1+\mu_{M_\nu^0})$  and $|\adj(I- \mu_{M_\nu^0}^{-1}M_\nu^0)z^\nu_0|^2$, where $z^\nu_0=\kappa_\nu^0 C_\nu^0\in\R^8$, are  rational functions of $(H^0_1,\dots,H^0_5)$ that are independent of the function $F$.  

\begin{lem}\label{MATLAB-lem-2} The rational functions of  $(H_1,\dots,H_5)$ representing $\mu_{M_\nu^0}(1+\mu_{M_\nu^0})$ and $|\adj(I- \mu_{M_\nu^0}^{-1}M_\nu^0)z^\nu_0|^2$ are not identically zero for each $\nu=1,\dots,5.$
\end{lem}
\begin{proof} Let $h_1(s)$ and $h_2(t)$ be the tensors defined   in the proof of Lemma \ref{MATLAB-lem-1}.  Using the {\em determinant} computations by MATLAB, one  shows that the values of the two rational functions  representing $\mu_{M_\nu^0}(1+\mu_{M_\nu^0})$ and $|\adj(I- \mu_{M_\nu^0}^{-1}M_\nu^0)z^\nu_0|^2$ are nonzero at  $[h_1(s),h_2(t),h_1(s),h_1(s),h_2(t)]$ with the same  $(s,t)$ and $\nu$ as given  in the proof of Lemma \ref{MATLAB-lem-1}.
 \end{proof}

Finally, we  construct  the function $F$ as follows.

\begin{defn}\label{def-F3} Let $F\colon \M^{2\times 2}\to \R$ be defined by  (\ref{def-F}) with the tensors $\tilde H_j=H_j^0-D^2 F_0(A_j^0)$ chosen to ensure  that
$(H_1^0,\dots,H_5^0)$ satisfies $(\ref{ImFT-1})$ and 
\begin{equation}
 \mu_{M_\nu^0}\notin \{-1,0\},\quad |\adj(I- \mu_{M_\nu^0}^{-1}M_\nu^0)z^\nu_0|^2 \ne 0.  \label{ImFT-41}
\end{equation}
The existence of such $(H_1^0,\dots,H_5^0)$'s  is guaranteed  by Lemmas \ref{MATLAB-lem-1} and \ref{MATLAB-lem-2}.  Note that,  by (\ref{def-F2}), $F$ is a smooth strongly polyconvex function.
\end{defn}

\begin{remk}\label{remk-F3} 
Since the exceptional values of $(H_1^0,\dots,H_5^0)$ not satisfying (\ref{ImFT-1}) or (\ref{ImFT-41}) are zeros of certain nonzero polynomials,  it follows that  the values of  $(H_1^0,\dots,H_5^0)=(D^2 F(A_1^0),\dots, D^2 F(A_5^0))$ satisfying (\ref{ImFT-1}) and (\ref{ImFT-41}) are   {\em  generic}  near $(D^2 F_0(A_1^0),\dots, D^2 F_0(A_5^0)).$ Such a  result has been derived in  \cite[Section 5]{Sz1}  from certain stability analysis of  $T_N$-configurations, which  may not  be available for  the special $T_N$-configurations due to possible  dimension deficiency. 
 \end{remk}

 \section{The open set $\Sigma$ and  verification of Condition (OC)}    
 
 Let $F$ be a function defined  in Definition \ref{def-F3} and $Y_\nu, z^\nu, M_\nu$ be the functions on $  B_\rho(P_\nu^0)$ defined  in Lemmas \ref{def-Y-nu} and  \ref{lem-key-1}.  Let $K=K_F.$
 
  By (\ref{ImFT-41}) and  continuity, we  select $\eta\in (0,\rho)$ sufficiently small so that 
  \begin{equation}\label{ImFT-6}
\begin{cases} \mu_{M_\nu(Q)}\notin \{-1,0\}, \\
\adj\Big [I- \mu_{M_\nu(Q)}^{-1}M_\nu(Q)\Big ]z^\nu(Q)\ne 0 \end{cases}  \forall\, Q\in \bar B_\eta (P_\nu^0)
\end{equation}
for all $\nu=1,\dots,5.$
Furthermore, from  $P_\nu ^1(Y_1(P_1^0),P_1^0)=P_\nu^0$  and  continuity,  we select $\beta\in (0,\eta)$  sufficiently small so that
\begin{equation}\label{ImFT-51}
P_\nu^1(Y_1(Q),Q)\in B_\eta(P_\nu^0) \quad \forall\; Q\in B_\beta(P_1^0) \quad \forall\,  \nu=1,\dots,  5.
\end{equation}
   
 \begin{lem}\label{rel-X-3} Let $\hat P^\nu_j (Q)=P^\nu_j(Y_\nu(Q),Q)$ for $Q\in B_\rho(P_\nu^0)$ and all $\nu, j\mod 5.$   Then 
\[
Y_1(Q)=Y_i(\hat P^1_i(Q)), \quad \hat P^i_{7-i} (\hat P^1_i(Q))=Q 
\]
for all $Q\in B_\beta(P_1^0)$ and $i=1,\dots,5.$ Therefore, in particular,
\begin{equation}\label{ImFT-8}
\det \frac{\partial \hat P^1_i}{\partial Q} (Q)\ne 0 \quad \forall\, Q\in B_\beta(P_1^0) \;\; \forall \, i=1,\dots,5.
\end{equation}
\end{lem}

  \begin{proof} Let $Q\in B_\beta(P_1^0).$ From $\Psi^1(Y_1(Q),Q)=0$ and (\ref{rel-X-4}), it follows that   $
  \Psi^{i}(Y_1(Q), P^1_{i}(Y_1(Q),Q))=0$ for each $i.$    Hence, by (\ref{ImFT-4}) and (\ref{ImFT-51}), one has $Y_1(Q)=Y_i(P^1_{i}(Y_1(Q),Q))=Y_i(\hat P_i^1(Q)).$ Thus, by definition of $P_j^\nu(Y,Q)$,
   \[
   \hat P_{7-i}^i(\hat P_i^1(Q))=P_{7-i}^i(Y_i(\hat P_i^1(Q)),\hat P_i^1(Q))=P_{7-i}^i(Y_1(Q),\hat P_i^1(Q))
   \]
   \[
   =\hat P_i^1(Q) +(C_i+C_{i+1}+\dots + C_5)|_{Y=Y_1(Q)}\]\[
   =Q+(C_1+\dots+C_{i-1}+C_i+\dots + C_5)|_{Y=Y_1(Q)}=Q.
   \]
   Finally, (\ref{ImFT-8}) follows from differentiating $\hat P^i_{7-i} (\hat P^1_i(Q))=Q$ with $Q$.
  \end{proof}

The following result  verifies  the Condition (OC) for  $\sigma=DF$ and thus proves Theorem \ref{thm-main-0} by virtue   of Yan \cite[Theorem 1.3 and Corollary 1.4]{Y}.
   
  \begin{thm}\label{thm-last}    Let 
$
 \hat X^\nu_j (Q)=Q+Z^\nu_j (Y_\nu(Q))$ for all $\nu, j=1,\dots,5$ and $Q\in B_\rho(P_\nu^0). $  
  Then   
$(\hat X_1^\nu(Q),\dots,\hat X_5^\nu(Q))\in M'_5\cap K_5$ for all $Q\in B_\rho(P_\nu^0).$    Define
  \begin{equation}\label{setSigma}
 \tilde \Sigma=\cup  \{T(\hat X^1_1(Q),\dots,\hat X^1_5(Q)) \colon  Q\in B_\beta(P^0_1)\},\quad \Sigma=\mathcal L^{-1}(\tilde \Sigma),
 \end{equation}
 where $\mathcal L\colon \M^{2\times 2}\times (\R^2)^2\to \M^{4\times 2}$  is the linear bijection defined by $(\ref{mapL})$.  
 Then  $\Sigma$  is nonempty, open and bounded, and  the function  $\sigma=DF$ satisfies the Condition (OC) with  open set $\Sigma.$ 
 \end{thm}
 
 \begin{proof}  Clearly, $(\hat X_1^\nu(Q),\dots,\hat X_5^\nu(Q))\in M'_5\cap K_5$ for all $Q\in B_\rho(P_\nu^0)$  and hence $\tilde \Sigma$ is well-defined, nonempty and  bounded; thus, so is  $\Sigma.$
Moreover,  one easily sees that  $\Sigma$  satisfies the condition (3.9) of  Definition 3.2 in \cite{Y} with  fixed  $N=5.$  
It remains to show that $\Sigma$ is open; this is equivalent to showing that   $\tilde \Sigma$ is open. So, let $\bar X\in \tilde \Sigma$ and assume
 \[
 \bar X =  \bar \lambda \hat X^1_i(\bar Q)+ (1-\bar \lambda)  \hat P^1_i(\bar Q)
 \]
    for some $\bar Q\in B_\beta(P_1^0)$, $i\in\{1,\dots,5\}$ and $0<\bar\lambda<1.$ The goal is to show that $B_\gamma(\bar X)\subset \tilde \Sigma$ for some $\gamma>0.$      (See Figure \ref{fig2} for a typical case of the proof.) 
       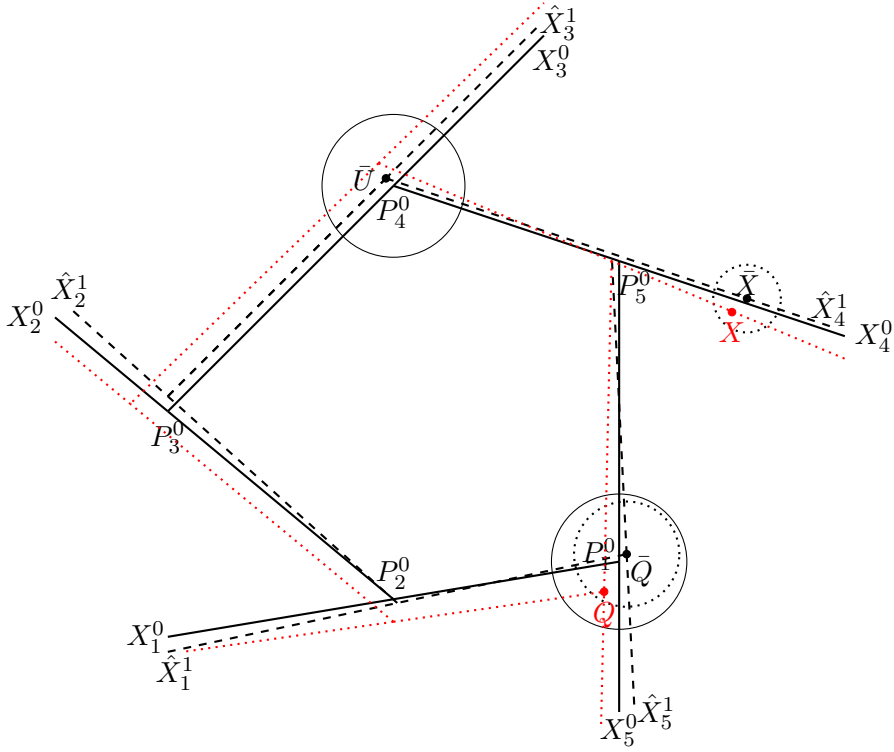
\begin{figure}[ht]
\begin{center}
\begin{tikzpicture}[scale =1]
\draw[thick] (-5,-2)--(1,-1);
\draw[thick,dashed] (-5,-2.2)--(1.1,-0.9);
  \draw(-4.9,-2) node[left]{$X^0_1$};
   \draw[](-4.9,-2.1) node[below]{$\hat X^1_1$};
    \draw[](0.2,5.8) node[above]{$\hat X^1_3$};
      \draw(0.1,6) node[below]{$X^0_3$};
   \draw(1.1,-0.9) node[left]{$P^0_1$};
    \draw(1.2,3) node[below]{$P^0_5$};
         \draw(-5,1) node[below]{$P^0_3$};
          \draw(-2,4) node[below]{$P^0_4$};
     \draw[thick] (1,-1)--(1,3);
      \draw[thick,dashed] (1.1,-0.9)--(0.9,3.1);
  \draw[thick] (1,-1)--(1,-3);
  \draw[thick,dashed] (1.1,-0.9)--(1.205,-3);
  \draw[thick] (-2,-1.5)--(-5,1);
   \draw[thick] (-6.5,2.25)--(-5,1);
   \draw[thick] (-5,1)--(-2,4);
   \draw[thick, dashed] (-5,1.2)--(-2.1,4.1);
    \draw[thick, dashed] (-5,1.2)--(-2.1,4.1);
     \draw[thick,dashed] (-1.95,-1.55)--(-6.3,1.2+2.75*1.3/3.05);
     \draw[thick] (0,6)--(-2,4);
     \draw[thick, dashed] (-0.1,6.1)--(-2.1,4.1);
    \draw[thick] (-2,4)--(1,3);
      \draw[thick, dashed] (-2.1,4.1)--(0.9,3.1);
       \draw[thick] (1,3)--(4,2);
         \draw[thick, dashed] (0.9,3.1)--(3.9,2.1);
     \draw(1,-2.9) node[below]{$X_5^0$};
      \draw[](1.1,-3) node[right]{$\hat X_5^1$};
         \draw(-2,-1.5) node[above]{$P^0_2$};
           \draw(1.3,-0.8) node[below, ]{$\bar Q$};
    \draw(4,2) node[right]{$X_4^0$};
     \draw[](3.8,2) node[above]{$\hat X_4^1$};
   \draw(-6.5,2.25) node[left]{$X^0_2$};
    \draw[](-6.3,2.25) node[above]{$\hat X^1_2$};
   \draw (1,-1) circle (0.9cm);
   \draw  (-2,4) circle (0.95cm);
     \draw[thick,dotted] (2.7,2.5) circle (0.45cm);
       \draw[](2.7,3) node[below]{$\bar X$};
        \draw[fill, ] (2.7,2.5) circle (0.05);
          \draw[fill, ] (1.1,-0.9) circle (0.05);
           \draw[thick,dotted] (1.1,-0.9) circle (0.7);
      
            \draw[fill, ]  (-2.1,4.1) circle (0.05);
         
      \draw[red,thick,dotted] (-2.2,4.3)--(4,1.7); 
       \draw[thick, red, dotted] (-2.2,4.3)--(-5.5,1.1); 
         \draw[thick, red, dotted] (-2.2,4.3)--(0,4.3+6.4/3); 
        \draw[thick, red, dotted] (-2,-1.8)--(-5.5,1.1); 
          \draw[thick, red, dotted] (-6.5,-1.8+26.1/7)--(-5.5,1.1); 
        \draw[thick, red, dotted] (-2,-1.8)--(0.8,-1.4); 
               \draw[thick, red, dotted] (-2,-1.8)--(-4.8,-2.2); 
                \draw[thick, red, dotted] (0.9,3)--(0.8,-1.4); 
                 \draw[thick, red, dotted] (0.76,-3.16)--(0.8,-1.4); 
\draw[fill, red] (0.8,-1.4)  circle (0.05); 
\draw[red] (0.8,-1.4) node[below] {$Q$}; 
   \draw[fill, red] (2.5,2.32)  circle (0.05);      
   \draw[red] (2.5,2.32) node[below] {$X$};  
            \draw[](-2.1,4.1) node[left]{$\bar U$};  
\end{tikzpicture}
\end{center}
\caption{A typical case of proving $\tilde \Sigma$ is open. Here $i=4$, $\bar Q\in B_\beta(P_1^0)$  and  $\bar U=\hat P_4^1(\bar Q)\in B_\eta(P_4^0).$ The  dashed lines represent the special $T_5$-configuration $(\hat X^1_1,\dots,\hat X^1_5)$ for which  $\bar X\in (\hat X^1_4,\bar U).$  The proof is to show   that every point $X$  in a small ball  $B_\gamma(\bar X)$ lies on the segment $(\hat X^1_4(Q),\hat P^1_4(Q))$ for some $Q\in  B_{\eta'}(\bar Q)\subset B_\beta(P_1^0).$} 
\label{fig2}
\end{figure}

Let $\bar U=\hat P^1_i(\bar Q);$ by (\ref{ImFT-51}),  $\bar U \in B_\eta(P_i^0).$   Let
$z(U)= z^{i}(U)=  Z^{i}_1(Y_{i}(U))$ be the function defined in Lemma \ref{lem-key-1}. By Lemmas \ref{rel-X} and \ref{rel-X-3},   we have
   \begin{equation}\label{eq-last-1}
  \bar X =  \hat P^1_i(\bar Q)+ \bar\lambda z(\hat P^1_i(\bar Q))= \bar U+ \bar\lambda z(\bar U).
  \end{equation}
  
  We proceed with the proof of $B_\gamma(\bar X)\subset \tilde \Sigma$  in two cases.
  
{\bf Case 1:} Assume  $\det (I+\bar\lambda Dz (\bar U))\ne 0.$ 
 In this case, we consider the function  $ G(Q, X)=\hat P_i^1(Q) +\bar\lambda z(\hat P_i^1(Q)) -X$ for  $Q\in  B_\beta(P^0_1)$ and all $X.$  By  (\ref{eq-last-1}),   
 one has $G(\bar Q, \bar X)=0;$ moreover, by (\ref{ImFT-8}),
  \[
 \det \frac{\partial G}{\partial Q}(\bar Q, \bar X) =  \det (I+\bar\lambda Dz (\bar U)  ) \cdot  \det \frac{\partial \hat P^1_i}{\partial Q} (\bar Q) \ne 0.
\]
  Therefore, by the implicit function theorem, there exist balls  $B_{\eta'}(\bar Q)\subset B_\beta(P_1^0)$ and $B_\gamma(\bar X)$ such that for each $X\in B_\gamma(\bar X)$  there exists  $Q\in B_{\eta'}(\bar Q)$ such that   $G(Q,X)=0$; hence 
  \[
  X= \hat P^1_i( Q)+ \bar\lambda z(\hat P^1_i(Q))=\bar \lambda \hat X^1_i(Q)+ (1-\bar \lambda)  \hat P^1_i(Q)\in \tilde\Sigma.
  \]
This proves $B_\gamma(\bar X)\subset \tilde\Sigma.$
 
{\bf Case 2:} Assume $\det (I+\bar\lambda Dz (\bar U))= 0$. In this case, 
let $\bar M=Dz(\bar U).$ Since $0<\bar\lambda <1,$ by Lemma \ref{lem-key-1}, one has  $\bar\lambda =- \mu_{\bar M}^{-1};$  thus, by (\ref{ImFT-6}),    
   $ \bar b= \adj \big (I- \mu_{\bar M}^{-1} \bar M \big ) z (\bar U)\ne 0.$ Consider $G(Q,X)=H(\hat P_i^1(Q), X),$ where
 \[
 H(U,X)= U+ \big ( \bar\lambda + (U-\bar U)\cdot \bar b\big ) z (U)-X
 \]
 for   $U\in  B_\rho(P^0_i)$ and all $X.$  Note that  $G(\bar Q,\bar X)=H(\bar U,\bar X)=0$ and 
 \[
 \frac{\partial H}{\partial  U} (\bar U,\bar X) =I+\bar \lambda \bar M +z(\bar U) \otimes \bar b.
 \]
Thus, by Lemma \ref{lem-key2}, $\det  \frac{\partial H}{\partial  U} (\bar U,\bar X)=[ \adj (I- \mu_{\bar M}^{-1} \bar M ) z(\bar U)] \cdot \bar b=|\bar b|^2\ne 0$, and hence, by (\ref{ImFT-8}),
\[
\det  \frac{\partial G}{\partial  Q} (\bar Q,\bar X)= \det  \frac{\partial H}{\partial  U} (\bar U,\bar X) \cdot \det \frac{\partial \hat P^1_i}{\partial Q} (\bar Q) \ne 0.
\]
The rest of the proof  follows similarly  as in Case 1. 

 This completes the proof.
 \end{proof}

\begin{remk}   For  $0<\lambda<1,$ let
 \begin{equation}\label{last-set}
 \tilde \Sigma_\lambda =\bigcup_{i=1}^5 \{ \lambda \hat X_i^1 (Q)+(1-\lambda)\hat P_i^1(Q):  Q\in B_\beta(P_1^0)\}.
 \end{equation}
By   (\ref{ImFT-6}),  $\det (I+ \lambda M_\nu (U))\ne 0$ for all $\nu$ and $U\in B_\eta(P_\nu^0)$ if $\lambda $ is sufficiently close to 1; thus, as in the Case 1 of the proof, it follows  that  the set $\tilde \Sigma_\lambda$ is open for all $\lambda $  sufficiently close to 1. Therefore,  by selecting a suitable sequence $\{\lambda_k\}\to 1$, the  open sets  $\{\tilde \Sigma_{\lambda_k}\}_{k=1}^\infty$  form   an {\em in-approximation} of  set $K_F.$ In this regard, it would be interesting to see whether   the convex integration scheme of  M\"uller  \& \v Sver\'ak \cite{MSv2} (see also \cite{KMS,MRS, Sz1})  
could be adapted for constructing Lipschitz but  nowhere $C^1$ (in space)   solutions  for the polyconvex gradient flow problem. 
  \end{remk}

\end{document}